\documentclass[USenglish]{article}	

\usepackage[utf8]{inputenc}				
\usepackage[big,online]{dgruyter}	
\usepackage{lmodern} 
\usepackage{microtype}
\usepackage{ dsfont }
\usepackage{subcaption}
\usepackage[numbers,square,sort&compress]{natbib}
\usepackage{siunitx}
\usepackage{xcolor}

\theoremstyle{dgthm}
\newtheorem{theorem}{Theorem}

\newtheorem{lemma}{Lemma}

\theoremstyle{dgdef}
\newtheorem{definition}{Definition}

\newtheorem{remark}{Remark}

\newtheorem{assumption}{Assumption}

\def\r{r}
\def\s{s}
\def\n{n}
\def\e{\epsilon}
\def\emin{{\e_{min}}}
\def\emax{{\e_{max}}}
\def\dE{\partial_\e}
\def\tE{{\triangle\e}}
\def\bdE{\bar\partial_\e}

\def\E{\mathcal{E}}

\def\R{\mathcal{R}}
\def\S{\mathcal{S}}
\def\T{\mathcal{T}}

\def\q{q}
\def\G{\mathcal{G}}

\def\d{\mathrm{d}}
\def\C{\mathcal{C}}
\def\A{\mathcal{A}}

\def\RR{\mathbb{R}}
\def\UU{\mathbb{U}}
\def\VV{\mathbb{V}}
\def\WW{\mathbb{W}}

\def\XX{\mathbb{X}}

\def\la{\langle}
\def\ra{\rangle}

\begin{document}


\title{Analysis and systematic discretization of a Fokker-Planck equation with Lorentz force}
\runningtitle{Structure-preserving discretization}

\author[1]{Vincent Bosboom}
\author*[2]{Herbert Egger}
\author[1]{Matthias Schlottbom} 
\runningauthor{V.~Bosboom et al.}
\affil[1]{\protect\raggedright 
University of Twente, Department of Applied Mathematics, Enschede, The Netherlands, e-mail: v.bosboom@utwente.nl, m.schlottbom@utwente.nl}
\affil[2]{\protect\raggedright 
Johannes Kepler Universit\"at, Department of Mathematics and Johann Radon Institute for Computational and Applied Mathematics, Linz, Austria, e-mail: herbert.egger@jku.at}


	
\abstract{The propagation of charged particles through a scattering medium in the presence of a magnetic field can be described by a Fokker-Planck equation with Lorentz force. This model is studied both, from a theoretical and a numerical point of view. A particular trace estimate is derived for the relevant function spaces to clarify the meaning of boundary values. Existence of a weak solution is then proven by the Rothe method. In the second step of our investigations, a fully practical discretization scheme is proposed based on an implicit Euler method for the energy variable
and a spherical-harmonics finite-element discretization with respect to the remaining variables. A complete error analysis of the resulting scheme is given and numerical tests are presented to illustrate the theoretical results and the performance of the proposed method. 
}


\maketitle
	
\section{Introduction}
The Boltzmann transport equation is a widely used model for the propagation of particles or radiation through scattering media \cite{Akkermans2007,Ishimaru1978,Modest2003}. In the forward-peaked regime, asymptotic analysis leads to the Fokker-Planck continuous slowing-down approximation~\cite{Chandrasekhar1943,Pomraning1992}.
This equation has been used for dose calculation in radiation therapy \cite{Hensel2006,Larsen1998} in order to describe the propagation of secondary electrons generated by inelastic scattering of a primary photon beam. 
In this paper, we consider an extension of the model that includes the effect of the Lorentz force on the electron distribution in the presence of a magnetic field, which is of interest in magnetic resonance imaging guided radiotherapy \cite{de_Pooter_2021,St-Aubin_2015,Vassiliev_2010}. 
In this context, the quasi-static distribution of secondary electrons propagating through a biological medium is described by 
\begin{alignat}{2} \label{eq:fp}
-\dE (S \psi) + \s \cdot \nabla_\r \psi + 
G \cdot \s \times \nabla_\s \psi 
- T \Delta_\s \psi
& = \q \qquad \text{on } \R \times \S \times \E.
\end{alignat}
Here $\psi=\psi(\r,\s,\e)$ is the phase-space density of electrons, depending on position $\r \in \R$, propagation direction $\s \in \S$, and energy level $\e \in \E=(\emin,\emax)$,  and $\q=\q(\r,\s,\e)$ is the source density.  
Furthermore, 
$\nabla_\r \psi$ denotes the spatial gradient, and $\nabla_\s$, $\Delta_\s$ the surface gradient and Laplace-Beltrami operator on the unit sphere; see \cite{Han2013,Pomraning1992} for parametric representations of these operators.
The coefficient $G=G(\r,\e)$ represents the scaled external magnetic field, while the parameters
$T=T(\r,\e)$ and $S=S(\r,\e)$ are derived from the scattering phase function in the forward-peaked regime \cite{Pomraning1992}.
Apart from the third term on the left-hand side of \eqref{eq:fp}, the equation can be found in \cite[Eq.~(14)]{Hensel2006};
for models including the Lorentz force, see \cite[Eq.~(11)]{Bouchard_2015} and \cite[Eq.~(10)]{St-Aubin_2015} as well as \cite{de_Pooter_2021,Swan_2021}.
The equation \eqref{eq:fp} is complemented by boundary conditions
\begin{alignat}{2}
\psi &= 0 \qquad && \text{on } \Gamma_{in} \times \E 
\quad \text{and} \quad 
\R \times \S \times \{\emax\},\label{eq:ibc}
\end{alignat}
which state that electrons can only leave but not enter the phase space $\R \times \S \times \E$. 
Using standard notation, we here decompose $\Gamma = \partial\R \times \S$ into an inflow and an outflow part
\begin{align} \label{eq:Gamma}
\Gamma_{in} &= \{(\r,\s) \in \partial\R \times \S:  n(r) \cdot \s < 0\}, \qquad \Gamma_{out} = \Gamma \setminus \overline{\Gamma_{in}},
\end{align}
with $n(r)$ denoting the outward unit normal vector on $\partial\R$. 
For ease of presentation, we only consider homogeneous boundary data, but the extension to inhomogeneous conditions is straightforward due to the linearity of the problem.  
Let us briefly discuss the main contributions obtained in this manuscript.

\smallskip 
\textbf{Existence of weak solutions.}
For vanishing magnetic field $G=0$ and spatially homogeneous stopping power $S=S(\e)$, the existence of a solution to \eqref{eq:fp}--\eqref{eq:ibc} can be deduced from the results in \cite{ShengHan2013,Herty2012}, which are based on earlier work \cite{Degond_1986,Degond_1987}. 
These papers consider \eqref{eq:fp} as a stationary problem in phase-space $\R \times \S \times \E$, and the existence proofs are based on Lions' representation theorem \cite{Arendt_2023,Showalter_1997}.
In this manuscript, we follow a different approach: We consider \eqref{eq:fp} as an evolution problem with respect to the energy $\e$, which is interpreted as a pseudo-time variable. Following the physical background, one moves from high to low energies, and the condition $\psi(\emax)=0$ in \eqref{eq:ibc} takes the role of an initial condition.   
We then use a Rothe method~\cite{Roubicek}: By an implicit discretization scheme, we construct a sequence of semi-discrete approximations, for which we establish uniform bounds in appropriate norms.
Existence of a solution can then be proven by weak compactness arguments.
This approach allows us to consider also spatially varying coefficients and non-vanishing magnetic fields.
Similar to \cite{ShengHan2013,Herty2012}, we also prove uniqueness in a class of regular solutions.
%

\smallskip 
\textbf{A trace theorem for the Fokker-Planck equation.}
The existence of boundary values for functions in anisotropic Sobolev spaces is a subtle issue. For the Boltzmann transport equation, the appropriate trace spaces are known \cite{Agoshkov_1998,Dautray6}. 
An additional technical difficulty arises for the Fokker-Planck approximation \eqref{eq:fp}, which seems to have been overlooked in some previous work. 
As part of our analysis, we thus provide a rigorous proof of a corresponding trace estimate, Lemma~\ref{lem:traces}, which might also be of independent interest.

\smallskip 
\textbf{Systematic discretization and error estimates.}
Various methods can be employed for the numerical solution of the Fokker-Planck equation~\eqref{eq:fp}. Monte-Carlo methods, see e.g. \cite{Bouchard_2015,Gifford2006}, are extremely flexible but pose computational challenges for applications like therapy planning, which require optimization with respect to model parameters~\cite{Frank2010}. Alternative methods based on deterministic discretization paradigms have therefore been considered; see for instance \cite{St-Aubin_2015,St-Aubin_2016,Swan_2021}. 
In this paper, we utilize a spherical-harmonics finite-element scheme, which has been proven successful in the context of neutron transport and radiative heat transfer; see e.g. \cite{Ackroyd_1997,Lewis_1984}. 
Together with the finite-difference approximation in energy, which was used to prove the existence of solutions on the continuous level, we obtain a fully practical discretization scheme with provable stability properties.
By extension of previous work \cite{ES2012,ES2016}, we perform a full discretization error analysis, which is further supported by numerical tests. 

\smallskip 
\textbf{Outline.}
The remainder of this article is organized as follows: 
In Section~\ref{sec:prelim}, we introduce some additional notation, our main assumptions, and some preliminary results. 
Section~\ref{sec:anal} is then concerned with the analysis of the problem. We establish the trace theorem, mentioned above, and prove existence of a weak solution.
In Section~\ref{sec:disc}, we introduce our fully discrete method and present its error analysis. 
For illustration of our theoretical results and the applicability of the method, some numerical tests are presented in Section~\ref{sec:num}.

\section{Preliminaries and Notation} 
\label{sec:prelim}

Throughout the manuscript, we make use of the following general assumptions on the problem data.

\begin{assumption}\label{ass:1}
$\R\subset\RR^3$ is a $\rho$-convex domain, $\S \subset \RR^3$ the unit sphere, and $\E=(\emin,\emax)$ a bounded interval. 
The parameter functions $T$, $S$ and $G_i$, $i=1,2,3$ lie in $W^{1,\infty}(\R\times\E)$. 
Moreover, the functions $T$ and $S$ are uniformly bounded from below, i.e., there exist constants $c_S$, $c_T>0$ such that $c_S \le S(\r,\e)$ and $c_T \le T(\r,\e)$ for a.e. $\r\in\R$ and $\e\in\E$. 
\end{assumption}
Since $\R$ is convex, its boundary is Lipschitz and the outward unit normal vector field $n$ is well-defined.
Bounds on the absolute value of a general function $F$ and its derivatives will be denoted by $C_{F}$ and $C_{F}'$, respectively. 
We use standard notation for function spaces, e.g. $L^p(\R \times \S)$ for the class of measurable functions whose $p$-th power is integrable or $C(\R \times \S \times \E)$ for continuous functions on $\R\times\S\times\E$. Furthermore, we use $L^p(\E;X)$ to denote the Bochner spaces of functions $f : \E \to X$ with values in some Banach space $X$.
For ease of notation, we introduce the abbreviations 
\begin{align*}
\la u,v\ra = \int_{\R \times \S}  u \, v\, \d(r,\s) 
\qquad \text{and} \qquad
\la u,v\ra_\partial = \int_{\partial\R \times \S} u \, v \, \d(r,\s) 
\end{align*}
for the scalar products of $L^2(\R \times \S)$ and $L^2(\partial\R \times S)$. 
The same symbols will be used later on also to denote duality products of certain Sobolev spaces, defined over the respective domains, and their dual spaces. 
By basic arguments, we obtain the following integration-by-parts formulas, which will be used later on. 
\begin{lemma} \label{lem:int-by-parts}
Let Assumption~\ref{ass:1} hold and $u,v \in C^2(\overline \R \times \S)$. Then
\begin{align*}
\la \s \cdot \nabla_\r u, v\ra &= -\la u, \s \cdot \nabla_\r v \ra + \la n \cdot \s \, u , v\ra_\partial \\
\la G \cdot \s \times \nabla_\s u, v \ra &= -\la u, G \cdot \s \times \nabla_\s v\ra \\
\la T \Delta_\s u, v\ra &= -\la T \nabla_\s u, \nabla_\s v\ra.
\end{align*}
For $u,v \in C^1(\overline \E;L^2(\R \times \S))$, we have 
\begin{align*}
\int\nolimits_\E \la \dE u, v \ra  \,\d\e &= -\int\nolimits_\E \la u, \dE v\ra \,  \d\e + \la u,v\ra\big|_{\emin}^{\emax}.  
\end{align*}
\end{lemma}
As a direct consequence, we obtain the following characterization of smooth solutions.
\begin{lemma}
Let $\psi$ be a smooth solution of \eqref{eq:fp}--\eqref{eq:ibc}. Then 
\begin{align} \label{eq:vp}
\int\nolimits_\E \la \psi, S \dE v \ra - \la \psi, \s \cdot \nabla_\r v \ra - \la \psi, G \cdot \s \times \nabla_\s v\ra 
+ \la T \nabla_\s \psi, \nabla_\s v\ra \, \d\e =\int\nolimits_\E \la \q, v \ra \,\d\e  %
\end{align}
for all smooth functions $v\in C^1(\overline\E \times \overline\R \times \S)$ with $v(\emin)=0$ and $v=0$ on $\Gamma_{out} \times \E$. 
\end{lemma}
The claim follows immediately by multiplying \eqref{eq:fp} with a smooth test function $v$, integrating over $\R \times \S \times \E$,  using the above integration-by-parts formulas, and the boundary conditions for $\psi$ and $v$.
This variational characterization of smooth solutions can be used to introduce the following solution concept.

\begin{definition}\label{def:weak_solution}
A function $\psi \in L^\infty(\E;L^2(\R \times \S))$ with $\nabla_\s \psi \in L^2(\E;L^2(\R \times \S))$ 
satisfying \eqref{eq:vp} 
for all $v\in C^1(\overline\R \times \S \times \overline \E)$ with $v(\emin)=0$ and $v=0$ on $\Gamma_{out} \times \E$, is called a weak solution of \eqref{eq:fp}--\eqref{eq:ibc}.
\end{definition}

Using the conditions of Assumption~\ref{ass:1}, existence of such a weak solution will be established next.

\section{Existence of solutions}
\label{sec:anal} 

The main goal of this section is to show the following generalization of corresponding results in \cite{ShengHan2013,Herty2012}.

\begin{theorem}\label{thm:existence}
Let Assumption~\ref{ass:1} hold. Then for any $\q \in L^2(\E;L^2(\R \times \S))$, there exists a weak solution $\psi$ of the system \eqref{eq:fp}--\eqref{eq:ibc} in the sense of Definition~\ref{def:weak_solution}.
\end{theorem}

The remainder of the section is devoted to the proof of this theorem. For orientation, let us briefly outline the main steps:
%
By backward differencing with respect to the energy variable $\e$, we construct a sequence of approximate solutions, and then prove uniform bounds on these approximations in appropriate spaces. Existence of a weak-solution is then obtained by weak-compactness arguments and linearity of the problem. 

\subsection{Energy discretization}

Let $\emax=\e^M > \e^{M-1} > \ldots > \e^0 = \emin$ denote a partition of the energy interval $\E=(\emin,\emax)$. For ease of notation, we assume $\e^{m-1} = \e^{m} - \tE$ to be equidistant.
For any sequence $(u^m)_{m \ge 0}$, we write
\begin{align*}
\bdE u^m = \frac{1}{\tE}(u^{m+1}-u^m)
\end{align*}
for the backward difference quotient. 
We use $u^m=u(\e^m)$ and $\bar u^m=\frac{1}{\tE} \int_{\e^{m}}^{\e^{m+1}} u(\e) \, \d\e$ to denote the evaluation and local averages of a function $u$ depending on the energy variable $\e$. 
Note that we traverse through \eqref{eq:fp} from high to low energy, following the physical origin of the slowing-down approximation. In view of \eqref{eq:ibc}, we thus choose $\psi^M=0$ for initialization. 
The approximations $\psi^m \approx \psi(e^m)$ for the lower energy levels $\e^m$, $m \le M-1$, are then obtained by solving recursively
\begin{alignat}{2} 
\label{eq:me}
-\bdE (S \psi)^m + \s \cdot \nabla_\r \psi^m + G^m \cdot \s \times \nabla_\s \psi^m - 
T^m \Delta_\s \psi^m &= \bar \q^m \qquad &&\text{in } \R \times \S, \\
\label{eq:me2}
\psi^m &= 0 \qquad &&\text{on } \Gamma_{in}.
\end{alignat}
Let us note that a local average of the source term is used on the right-hand side of \eqref{eq:me}.
Apart from this modification and the reverse transition through the energy levels, from high to low, this method amounts to a standard implicit Euler \emph{time-stepping} scheme, with $\e$  interpreted as pseudo-time.

\subsection{A trace theorem}

Extending the considerations of \cite{Agoshkov_1998,ES2012}, the natural Hilbert spaces for the analysis of \eqref{eq:me}--\eqref{eq:me2} turn out to be 
\begin{align}
\VV &= \{v \in L^2(\R \times \S) : \nabla_\s v \in L^2(\R \times \S)\},  \label{eq:def_V}\\
\WW &= \{v \in \VV: \s \cdot \nabla_\r v \in \VV^*,\, |s\cdot n|^{1/2}v\in L^2(\Gamma)\}, \label{eq:def_W}
\end{align}
with $\VV^*$ denoting the dual space of $\VV$, where the norm on $\VV$ is given by $\|\cdot\|_{\VV}^2=\|\cdot\|_{L^2(\R \times \S)}^2+\|\nabla_\s \cdot\|_{L^2(\R \times \S)}^2$.
In order to verify that the definition of $\WW$ makes sense, one has to ensure that functions $v \in \VV$ with directional derivatives $\s \cdot \nabla_r v \in \VV^*$ have well-defined traces. 
This can be guaranteed by the following technical result. 
\begin{lemma}[Trace estimate] \label{lem:traces}
Let $\R,\S$ satisfy the conditions of Assumption~\ref{ass:1} and $\Gamma_{in}$ be defined as in \eqref{eq:Gamma}.
Then there exists a constant $C>0$, depending only on $\R$, such that for all $v\in\VV$ with $\s\cdot\nabla_\r v\in\VV^*$, one has
\begin{align*}
    \int_{\Gamma_{in}} |v|^2  |\s\cdot n|\tau^2 \d(\r,\s) \leq C \big(\| \s\cdot\nabla_\r v\|_{\VV^*}^2 + \|v\|_{L^2(\R\times\S)}^2 \big)^{1/2}\|v\|_{\VV}.
\end{align*}
Here $\tau=\tau(\r,\s)$ is the length of the intersection of $\R$ with the line $t\mapsto \r+t\s$. 
\end{lemma}
\begin{proof}
We adapt the proof of \cite[Thm.~2.2]{Manteuffel1999}. 
Let $\Gamma_{in}(\s)=\{\r\in \partial\R: n(\r)\cdot \s<0\}$ and $\Gamma_{out}(\s)=\partial\R \setminus \overline{\Gamma_{in}(s)}$ 
be the inflow and the outflow part of $\partial \R$ for a fixed direction $\s\in\S$. 
We split $\tau=\tau_-+\tau_+$, where $\tau_-$ is the distance along the line segment from $\r$ to the inflow boundary $\Gamma_{in}(s)$, while $\tau_+$ is the corresponding distance to the outflow boundary.
We further define $z(\r,\s)=1-\tau_-(\r,\s)/\tau(\r,\s)$, and observe that $z(\r,\s)=1$ for $r\in \Gamma_{in}(s)$ and $z(\r,\s)=0$ for $r\in\Gamma_{out}(s)$.
For $r\in\Gamma_{in}(s)$, we then see that $z(\r+t\s,\s)=1-t/\tau(\r,\s)$ and $s\cdot\nabla_r z(\r+t\s,\s)=-1/\tau(\r,\s)$.
By the fundamental theorem of calculus, we then compute for $\r\in\Gamma_{in}(s)$
\begin{align*}
    v(\r,\s)^2 &= (v(\r,\s)z(\r,\s))^2= - \int_0^{\tau(\r,\s)} s\cdot\nabla_r (v(\r+t\s,\s)z(\r+t\s,\s))^2 \,\d t\\
    &= - 2 \int_0^{\tau(\r,\s)} s\cdot\nabla_\r v(\r+t\s,\s) v(r+ts,s) \left(\frac{\tau(\r,\s)-t}{\tau(\r,\s)}\right)^2
    -  v(\r+t\s,\s)^2  \frac{\tau(\r,\s)-t}{\tau(\r,\s)^2}\,\d t,
\end{align*}
for any $v \in C^1(\overline \R \times \S)$. 
Multiplying the latter identity by $|s\cdot n|\tau(\r,\s)^2$ and integrating over $\Gamma_{in}(s)$ yields
\begin{align*}
    \int_{\Gamma_{in}(s)}|v|^2|s\cdot n|\tau(\r,\s)^2 \,\d\r = - 2 \int_{\Gamma_{in}(s)}\int_0^{\tau(\r,\s)} &\Big( s\cdot\nabla_r v(\r+t\s,\s) v(r+ts,s) (\tau-t)^2 \\
    - & v(r+t\s,\s)^2  (\tau-t)\Big)|s\cdot n|\,\d t\,\d\r.
\end{align*}
By integration over $\S$ and using the  identity
$\int_{\Gamma_{in}(s)}\int_0^{\tau(\r,\s)} f(\r+t\s) |\s\cdot n| \d t\d\r = \int_{\R} f(\r)\d\r$, which holds for any $f\in L^1(\R)$, see for instance in \cite[Lem.~1]{Stefanov:1999}, we then immediately obtain the identity
\begin{align*}
    \int_{\Gamma_{in}}|v|^2|s\cdot n|\tau(\r,\s)^2 \,\d(\r,\s) 
    = - 2 \int_{\R\times\S}  s\cdot\nabla_r v v \tau_+^2
    -  v^2  \tau_+\,\d(\r,\s).
\end{align*}
An application of the Cauchy-Schwarz inequality now shows that
\begin{align*}
    \int_{\Gamma_{in}}|v|^2|s\cdot n|\tau(\r,\s)^2 \,\d(\r,\s) 
    \leq  2  \|s\cdot\nabla_r v\|_{\VV^*} \|v \tau_+^2\|_{\VV} +2 \|v \tau_+^{1/2}\|_{L^2(\R\times\S)}^2.
\end{align*}
To estimate the last term, we use that $\tau_+ \leq {\rm diam}(\R)$ and that $\nabla_s\tau_+$ is bounded because $\R$ is $\rho$-convex, see \cite{ES2026}.
Therefore, $\|v \tau_+^2\|_{\VV}\leq C \|v\|_{\VV}$ with a constant depending on $\R$. This shows the validity of the bounds for smooth functions, and the claim of the lemma finally follows by a density argument.
\end{proof}

\subsection{Well-posedness of the semi-discrete scheme}
\label{sec:inf_sup}

Due to Lemma~\ref{lem:traces}, the Hilbert space $\WW$ with norm 
$\|w\|_{\WW}^2=\|w\|_\VV^2+\|\s \cdot \nabla_\r w\|_{\VV^*}^2 + \||\s \cdot n|^{1/2} w \|_{L^2(\Gamma)}^2$ 
and corresponding inner product is well-defined.
As a next step, we introduce some abbreviations for the differential operators appearing in \eqref{eq:fp}, namely  
\begin{align*}
\A u = \s \cdot \nabla_r u 
\qquad \text{and} \qquad 
\G u = G(\e^m) \cdot \s \times \nabla_\s u \qquad \forall u \in \WW. 
\end{align*}
For the surface Laplacian, we apply integration by parts and use a weak characterization, i.e., 
\begin{alignat*}{2}
\la \T u, v \ra = \la  T(\e^m) \nabla_\s u, \nabla_\s v \ra \qquad \forall u,v \in \WW.
\end{alignat*}
Note that $\G$ and $\T$ implicitly depend on the time step $m$, and we write $\G^m$  and $\T^m$ to indicate this dependence, if required. 
Similarly, we denote by $\S^m$ the multiplication operator related to the stopping power $S(\e^m)$.
Following \cite{ES2012}, we further decompose functions of the angular variable via
\begin{align*}
\psi = \psi^+ + \psi^- \qquad \text{with} \qquad 
\psi^{\pm}(\s) = \frac{1}{2}(\psi(\s) \pm \psi(-\s))
\end{align*}
into \emph{even} and \emph{odd} parts.
This decomposition is $L^2(\S)$-orthogonal, and it carries over to functions in $\VV$ and $\WW$. 
We hence denote by $\VV^+$ and $\WW^+$ functions in $\VV$ respectively $\WW$ that are even in the $\s$-variable. For the corresponding subspaces of odd functions we write $\VV^-$ and $\WW^-$, respectively.
Hence, we can identify $(w^+,w^-)\in\WW^+ \times \VV^-$ with $w=w^+ + w^-$, and we write $\WW^+ \oplus \VV^-$ for the topological direct sum of $\WW^+$ and $\VV^-$ with inherited norm $\|w\|_{\WW^+\oplus\VV^-}^2=\|w^+\|_{\WW}^2+\|w^-\|_{\VV}^2$.
We then define the mixed regularity space $\UU$ as the set $\WW^+\oplus\VV^-$ endowed with the norm
\begin{align} \label{eq:normU}
\|u\|_{\UU}^2 =\|u\|_{\C}^2+\|u^+\|_\partial^2 + \|\A u^+\|_{\C^{-1}}^2,
\end{align}
where $\|u\|^2_{\C}=\la \C u,u\ra$ and $\|u\|_\partial^2=\la |s\cdot n| u,u\ra_\partial$ with generalized collision operator $\C=\frac{1}{\tE} \S^m +\T$.
By Assumption~\ref{ass:1}, this norm is equivalent to the natural norm on $\WW^+\oplus \VV^-$, and thus $\UU$ is a Hilbert space.
Using elementary arguments, see \cite{ES2012}, one can then verify the following observation.
\begin{lemma}\label{lem:var}
Let $\psi^m \in \WW$ with $\Delta_\s\psi^m,\s\cdot\nabla_\r\psi^m\in L^2(\R\times\S)$ be a solution of \eqref{eq:me}--\eqref{eq:me2} for given data $\psi^{m+1} \in \VV$ and $\bar \q^m \in L^2(\R \times \S)$. 
Then
\begin{align}\label{eq:var_eq}
 -\la \bdE(S\psi)^m,v\ra + a(\psi^m,v) = \la \bar \q^m,v\ra \qquad \forall v \in \UU
\end{align}
with bilinear form $a:\UU\times\UU\to\RR$ defined by
\begin{align}\label{eq:def_bifo}
    a(u,v)=\la \G u,v\ra + \la |\s\cdot n|u^+,v^+\ra_\partial + \la \A u^+,v^-\ra -\la u^-,\A v^+\ra +\la \T u,v\ra,\qquad \forall u,v\in\UU.
\end{align}
\end{lemma}
\begin{proof}
We proceed similarly to~\cite{ES2012}.
    Multiplication of \eqref{eq:me} with $v\in\UU$ and integration over $\R\times\S$ yields
    \begin{align}\label{eq:me_tested}
        -\la \bdE (S \psi)^m,v\ra + \la\s \cdot \nabla_\r \psi^m,v\ra + \la G^m \cdot \s \times \nabla_\s \psi^m ,v\ra - \la T^m \Delta_\s \psi^m,v\ra = \la \bar \q^m ,v\ra.
    \end{align}
    We next apply Lemma~\ref{lem:int-by-parts} to see that $- \la T^m \Delta_\s \psi^m,v\ra=\la T^m \nabla_\s \psi^m,\nabla_\s v\ra$. It thus remains to investigate the term $\la\s \cdot \nabla_\r \psi^m,v\ra$.
    By the orthogonality of even and odd functions and Lemma~\ref{lem:int-by-parts}, we observe that
    \begin{align*}
         \la\s \cdot \nabla_\r \psi^m,v\ra 
         =\la\s \cdot \nabla_\r \psi^{m,+},v^-\ra - \la\psi^{m,-}, \s \cdot \nabla_\r v^+\ra +\la \s\cdot n \psi^{m,-},v^+\ra_\partial.
    \end{align*}
    Noting that $\s\cdot n \psi^{m,-} v^+$ is an even function of $\s$, we can now further deduce that
    \begin{align*}
        \la \s\cdot n \psi^{m,-},v^+\ra_\partial =2 \la \s\cdot n \psi^{m,-},v^+\ra_{\Gamma_{in}} = 2 \la |\s\cdot n| \psi^{m,+},v^+\ra_{\Gamma_{in}} = \la |\s\cdot n| \psi^{m,+},v^+\ra_\partial,
    \end{align*}
    where we used the boundary condition $\psi^{m,-}=-\psi^{m,+}$ and $\s\cdot n=-|\s\cdot n|$ on $\Gamma_{in}$ in the second step, and that $|\s\cdot n| \psi^{m,+}v^+$ is an even function in the third step. Thus, $\la\s \cdot \nabla_\r \psi^m,v\ra=\la\s \cdot \nabla_\r \psi^{m,+},v^-\ra - \la\psi^{m,-}, \s \cdot \nabla_\r v^+\ra +\la |\s\cdot n| \psi^{m,+},v^+\ra_\partial$. Using this identity in \eqref{eq:me_tested} completes the proof.
\end{proof}
Let us note that the variational identity \eqref{eq:var_eq} makes sense for functions $\psi^m \in \UU$, and we accept such functions as solutions for \eqref{eq:me}--\eqref{eq:me2}. Under Assumption~\ref{ass:1}, the existence of such solutions can be established.
\begin{lemma}\label{lem:well_posed_euler}
For any $\bar \q^m,\psi^{m+1} \in L^2(\R \times \S)$, the system \eqref{eq:me}--\eqref{eq:me2} has a unique solution $\psi^m \in \UU$.
\end{lemma}
\begin{proof}
We closely follow the arguments of \cite{ES2012} and, therefore, stay very brief in the sequel. 
By a slight rearrangement of terms, one can see that \eqref{eq:var_eq} is equivalent to the problem
\begin{align}\label{eq:var3}
    b(u,v)=\ell(v) \qquad \forall v\in\UU
\end{align}
with solution $u=\psi^m$, bilinear form $b(u,v)= \frac{1}{\tE} \la \S^m u,v\ra + a(u,v)$, and $\ell(v) = \la \bar \q^m,v\ra + \frac{1}{\tE} \la \S^{m+1} \psi^{m+1},v\ra$ abbreviating 
the right-hand side. 
It is not difficult to verify that $b :\UU \times  \UU \to \RR$ is bilinear and continuous, and that $\ell : \UU \to \RR$ is linear and continuous. 
From the integration-by-parts formulas of Lemma~\ref{lem:int-by-parts}, one can further deduce that $\la \G v,v\ra=0$ for all $v \in \VV$. This immediately implies 
\begin{align*}
b(u,u) &= \|u\|^2_{\C} + \|u^+\|_{\partial}^2.
\end{align*}
Choosing $v=\C^{-1} \A u^+$ as a test functions and observing that $v$ and $\G u^-$ are odd functions, we further get
\begin{align*}
b(u,\C^{-1} \A u^+) 
&= \langle (\C+\G) u^-, \C^{-1} \A u^+\rangle + \langle \A u^+, \C^{-1} \A u^+\rangle\\
&\ge -\frac{1}{2}\|(\C+\G) u^-\|_{\C^{-1}}^2 + \frac{1}{2} \|\A u^+\|_{\C^{-1}}^2
 \geq -\frac{1}{2}(1+C_\G^2)\|u^-\|_\C^2 +\frac{1}{2}\|\A u^+\|_{\C^{-1}}^2.
\end{align*}
Here we used Young's inequality, the basic identity
\begin{align*}
    \|(\C+\G) u^-\|_{\C^{-1}}^2=\langle u^-,\C u^- \rangle +2\langle \G u^-,u^- \rangle + \langle \C^{-1} \G u^-, \G u^- \rangle = \|u^-\|_{\C}^2+\|\G u^-\|_{\C^{-1}}^2,
\end{align*}
which follows from $\langle \G u^-,u^- \rangle=0$ by Lemma~\ref{lem:int-by-parts},
as well as the bound $\|\G u^-\|_{\C^{-1}}\leq C_\G \|u^-\|_{\C}$. 
The latter estimate follows from the bounds for $G$ and $T$ and elementary properties of the operators.
Setting $v=u + \gamma \C^{-1} \A u^+$ with $\gamma=2/(2+C_\G^2)$, we thus obtain $v^+=u^+$ and $v^-=u^- + \gamma \C^{-1}\A u^+$ and
\begin{align} \label{eq:infsup}
b(u,v) 
\ge  \|u^+\|_{\C}^2 +  \frac{\gamma}{2} \|u^-\|^2_{\C} +  \frac{\gamma}{2} \|\A u^+\|_{\C^{-1}}^2 + \|u^+\|_{\partial}^2
\ge \frac{\gamma}{2}\|u\|_{\UU}^2.
\end{align}
Using the test function $v=u-\gamma \C^{-1}\A u^+$, one can show $b(v,u) \ge   \frac{\gamma}{2} \|u\|_{\UU}^2$ in a similar manner.
Furthermore
\begin{align}\label{eq:bound_test_fct_inf_sup}
\|v\|_{\UU} =\|u  \pm \gamma \C^{-1} \A u^+\|_{\UU} \le C_A \|u\|_{\UU},
\end{align}
for some positive constant $C_A>0$ independent of $u$.
These inequalities verify the stability conditions of the Babuska-Aziz lemma~\cite{Babuska_1971}, and we can thus conclude the existence of a unique solution $u\in \UU$ of our variational problem together with an a-priori bound $\|u\|_{\UU} \le C \|\ell\|_{\UU^*} \le C' (\|\bar \q^m\|_{L^2(\R \times \S)} + \|\psi^{m+1}\|_{L^2(\R \times \S)})$.
\end{proof}
This result clarifies the well-posedness of \eqref{eq:me}--\eqref{eq:me2} for a single time step. By induction over $m$ and noting that $\UU \subset L^2(\R \times \S)$, we then obtain existence of a semi-discrete solution $\psi^m$, $0 \le m \le M$ in $\UU$. 

\subsection{Uniform bounds}

The constants of the a-priori bounds in the last step of the previous proof depend on the step size parameter. In the following, we show that the semi-discrete approximation can be bounded independent of $\tE$. To this end, we mimick the basic identity 
\begin{align*}
\dE (S \psi^2) = 2 \dE (S \psi) \, \psi - (\dE S) \psi^2,  
\end{align*}
which follows immediately by the product rule of differentiation. 
A corresponding discrete version reads
\begin{align} \label{eq:product}
\bdE (S \psi^2)^{m} = 2 (\bdE (S\psi))^m \psi^{m} - (\bdE S)^m |\psi^m|^2 + \tE S^{m+1} |\bdE \psi^m|^2.
\end{align}
The last term here stems from the dissipative nature of the backward difference quotient. 
We can now prove the following a-priori bounds, which will allow us to establish the existence of a weak solution later on.
\begin{lemma}\label{lem:stability}
Let Assumption~\ref{ass:1} hold and $\psi^m$, $0 \le m \le M$ denote a solution of \eqref{eq:me}--\eqref{eq:me2} with $\tE$ sufficiently small, i.e., such that 
$0 < \tE\leq \frac{c_S}{2 (C_S'+1)}$. 
Then the estimate 
\begin{align} \label{eq:stability}
    \sup_{0\leq m\leq M}\|\psi^m\|_{L^2(\R\times\S)}^2  + \sum_{m=0}^{M-1} \tE \|\nabla_\s\psi^m\|_{L^2(\R\times\S)}^2\leq C \, \|\q\|_{L^2(\E;L^2(\R \times \S))}^2
\end{align}
holds with a constant $C$ that is independent of the  step size parameter $\tE$.
\end{lemma}
\begin{proof}
Solutions of \eqref{eq:me}--\eqref{eq:me2} are characterized by \eqref{eq:var_eq}. When testing this identity with $v=\psi^{m}$, we obtain
\begin{align*}
    -\la \bdE(S^m\psi^m),\psi^m\ra + \la|s\cdot n|\psi^{m,+},\psi^{m,+}\ra_\partial+\la T^m\nabla_\s\psi^{m},\nabla_\s\psi^{m}\ra
    =  \la \bar \q^m,\psi^m\ra.
\end{align*}
Note that some of the terms appearing in $a(\psi^m,\psi^m)$ vanish due to anti-symmetry. 
With the help of the identity \eqref{eq:product}, we may rewrite the term involving $\bdE(S^m\psi^m)$ as 
\begin{align*}
    -2\tE\la  \bdE(S^m\psi^m),\psi^m\ra
    &= 
     \la S^{m}\psi^{m},\psi^{m}\ra-\la S^{m+1}\psi^{m+1},\psi^{m+1}\ra
    + \la (S^m-S^{m+1})\psi^m,\psi^m\ra
    \\
    & \qquad \qquad + \la S^{m+1}(\psi^{m+1}-\psi^{m}),(\psi^{m+1}-\psi^{m})\ra. 
\end{align*}
The last term on the right-hand side is positive, and the third term can be bounded by
\begin{align}\label{eq:bound}
    \la (S^m-S^{m+1})\psi^m,\psi^m\ra\geq -\tE  C_S' c_S^{-1} \la S^{m}\psi^{m},\psi^{m}\ra,
\end{align}
where we used the upper and lower bounds on $S'$ and $S$ provided by Assumption~\ref{ass:1}. 
For abbreviation, we introduce the new constant $\tilde C_S = C_S'/c_S$. A combination of the previous estimates leads to
\begin{align*}
   \Big( 1-\tE \, \tilde C_S\Big) \la S^m\psi^m,&\psi^m\ra + 2\tE \la T^m\nabla_\s\psi^{m},\nabla_\s\psi^{m}\ra\\
   &
    \leq \la S^{m+1}\psi^{m+1},\psi^{m+1}\ra + 2\tE  \la \bar \q^m,\psi^m\ra \\
   & \leq \la S^{m+1}\psi^{m+1},\psi^{m+1}\ra +  \tE\la \bar \q^m,\bar \q^m\ra +  \tE \, c_S^{-1} \la S^m\psi^m,\psi^m\ra.
\end{align*}
Using that $\tE \le c_S / (2 (C_S'+1))$, 
the leading term on the left-hand side can be bounded from below by the positive constant $1-\tE (\tilde C_S+ c_S^{-1})$.
We may then apply this inequality recursively, to see that 
\begin{align}\label{eq:energy}
\la S^m\psi^m,\psi^m\ra +\sum_{k=m}^{M-1} \tE \la T^k \nabla_\s\psi^k,\nabla_\s\psi^k\ra
    \leq \widehat C_S \sum_{k=m}^{M-1}\tE \|\bar \q^k\|^2_{L^2(\R\times\S)}.
\end{align}
The constant $\widehat C_S$ only depends on $c_S$, $C_S'$ and the size $|\E|$ of the energy interval. 
The assertion of the lemma now follows by observing that $\sum_{k=0}^{M-1} \tE \|\bar \q^k\|^2_{L^2(\R \times \S)} \le \|\q\|_{L^2(\E;L^2(\R \times \S))}^2$, which follows immediately from the definition of the local averages $\bar \q^k$, and noting that $S$ and $T$ are uniformly positive.
\end{proof}

\begin{remark}\label{rem:local_de}
   As shown in Lemma~\ref{lem:well_posed_euler}, the semi-discrete solution $\psi^m$ is well-defined for arbitrary $\tE>0$, which could be chosen differently for every step $m$.
   Also the inequality
   \eqref{eq:bound} and the ones thereafter remain valid 
   under a local restriction on the step size, e.g.
    $2 \tE^m\leq 1/\sup_{\r\in\R} \big( (|\bar S'(\r,\epsilon^m)|+1)/S(r,\epsilon^m)\big)$. 
    Here $\bar S'(\r,\epsilon^m)$ denotes the average of $S'(\r,\epsilon)$ over the interval $[\epsilon^{m},\epsilon^{m+1}]$.
    The stability estimate of Lemma~\ref{lem:stability} thus generalizes quite naturally to adaptive time steps $\tE^m$ with appropriate local restrictions on the step size.
\end{remark}

\subsection{Proof of existence}
Let $\psi^m$, $0 \le m \le M$ denote a solution of the energy stepping procedure \eqref{eq:me}--\eqref{eq:me2} with step size $\tE$ as constructed in Lemma~\ref{lem:well_posed_euler}. Then we define a piecewise constant extension $\psi_{\tE} \in L^2(\E;\UU)$ such that $\psi_{\tE}(\e) = \psi^{m}$ for $\e \in (\e^{m-1},\e^m]$. 
From the uniform bounds of the previous lemma, we now conclude that
\begin{align*}
\|\psi_{\tE}\|_{L^\infty(\E;L^2(\R\times\S))} + \|\nabla_\s \psi_{\tE}\|_{L^2(\E;L^2(\R \times \S))} \le C,
\end{align*}
with a uniform constant $C$ independent of $\tE$. 
By the Banach-Alaoglou theorem \cite[p.~66]{Brezis2011}, we may thus select a sequence of functions $\psi_{\tE}$ for different values of $\tE$, and a limit $\psi \in L^\infty(\E;L^2(\R\times\S))$ with derivative $\nabla_\s \psi  \in L^2(\E;L^2(\R\times\S))$, such that
\begin{alignat*}{2}
    \psi_{\tE} &\rightharpoonup^* \psi \qquad && \text{in } L^\infty(\E,L^2(\R\times\S))\\
    \psi_{\tE} &\rightharpoonup \psi \qquad &&\text{in } L^2(\E,L^2(\R\times\S))\\
    \nabla_\s\psi_{\tE} &\rightharpoonup \nabla_s \psi \qquad &&\text{in } L^2(\E,L^2(\R\times\S))
\end{alignat*}
with step size $\tE \to 0$.
We will now show that $ \psi$ is a weak solution to \eqref{eq:fp}--\eqref{eq:ibc} in the sense of Definition~\ref{def:weak_solution}.
Let $v\in C^1(\overline\R \times \S \times \overline \E)$ be a smooth test function with $v(\emin)=0$ and $v=0$ on $\Gamma_{out} \times \E$. 

\smallskip 
\noindent 
\textit{Step~1.} 
By definition of the extension $\psi_\tE$, we see that 
\begin{align*}
-\sum_{m=0}^{M-1} \tE \la \bdE (S \psi)^m, v^m\ra 
&= \sum_{m=1}^{M} \la S^m \psi^m, v^m - v^{m-1}\ra \\
&= \sum_{m=1}^{M} \int\nolimits_{\e^{m-1}}^{\e^m} \la S_\tE \psi_\tE , \dE v\ra \d\e 
\to \int\nolimits_\E \la S \psi, \dE v \ra \d\e \quad\text{as }\tE\to 0.
\end{align*}
Here we used that $\psi^M=0$ and $v^0=v(\emin)=0$ by assumption, and we denoted by $S_\tE$ the piecewise constant approximation of $S$ with $S_\tE(\e) = S(\e^m)$ for $\e \in (\e^{m-1},\e^{m}]$. 
Let us further note that the difference $\|\S_{\tE} - S\|_{L^\infty(\E;L^\infty(\R))} \to 0$ by Assumption~\ref{ass:1}, which yields the convergence in the last step.

\smallskip 
\noindent
\textit{Step~2.}
Using integration-by-parts and the boundary conditions for $v$, one can show that 
\begin{align*}
\sum_{m=0}^{M-1} \tE \Big( &\la \s \cdot \nabla_\r \psi^{m,+}, v^{m,-}\ra  + \la |\s \cdot n| \psi^{m,+}, v^{m,+}\ra_\partial - \la \psi^{m,-},\s\cdot\nabla_\r v^{m,+}\ra \Big) \\
&= -\sum_{m=0}^{M-1} \tE  \la \psi^m, \s \cdot \nabla_\r v^m \ra \,
\to -\int\nolimits_\E \la \psi, \s \cdot \nabla_\r v \ra \, \d\e \quad\text{as }\tE\to 0.
\end{align*}
For the first equality, we used the same arguments as in the derivation of the variational principle \eqref{eq:var_eq}; for details, let us refer to \cite{ES2012}.
This observation thus handles the spatial derivative terms. 

\smallskip 
\noindent 
\textit{Step~3.}
The convergence of the remaining terms in the definition of a weak solution follows immediately from their definition and the weak convergence of the functions $\psi_\tE$ stated above.  

\smallskip 
\noindent 
By adding up the contributions and using \eqref{eq:me}--\eqref{eq:me2}, we see that the limit function $\psi$ satisfies \eqref{eq:vp}. \qed 

\subsection{Uniqueness of regular solutions}\label{sec:uniqueness}
Theorem~\ref{thm:existence} guarantees the existence of a weak solution to \eqref{eq:vp}. 
For completeness, we now also show uniqueness of weak solutions $\psi$ that satisfy the extra regularity condition $\s\cdot\nabla_\r\psi(\epsilon) \in L^2(\R\times\S)$ for a.e. $\epsilon \in\E$. 
To the best of our knowledge, uniqueness of weak solutions without extra regularity remains an open question.
This is inline with other existence results, such as \cite{Herty2012}, that rely on Lions' representation theorem \cite{Lions61}, which does not guarantee uniqueness, see also \cite[III. Theorem 2.1]{Showalter_1997}.
To proceed, let us introduce the space
$$
\XX=\{v\in\VV: \s\cdot\nabla_\r v\in L^2(\R\times\S)\}.
$$
A weak solution $\phi$ to our problem is called \emph{regular}, if $\phi \in L^2(\E;\XX)$. We will show uniqueness of such regular weak solutions.
By linearity of the problem, it suffices to prove the following.
\begin{lemma}
Let Assumption~\ref{ass:1} hold and $\psi$ be a weak solution of \eqref{eq:fp}--\eqref{eq:ibc} for $q=0$ in the sense of Definition~\ref{def:weak_solution}. 
Further assume that $\phi$ is regular, i.e, $\psi \in L^2(\E;\XX)$. Then $\psi=0$.
\end{lemma}
\begin{proof}
By testing \eqref{eq:vp} with functions $v$ having compact support in $\R \times \S \times \E$, using integration-by-parts, and the additional regularity of $\psi$, one can see that 
\begin{align} \label{eq:strong}
\int\nolimits_\E \langle \psi, S \dE v \rangle \d \e 
&= -\int\nolimits_\E \langle \s \cdot \nabla_\r \psi, v \rangle +  \langle
G \cdot \s \times \nabla_\s \psi, v \rangle + \langle T  \nabla_\s \psi, \nabla_\s v \rangle \, \d \e. 
\end{align}
All terms on the right-hand side are well-defined for $v \in C_0^\infty(\E;\VV)$, which shows that $S \psi$ has a weak derivative $\dE (S \psi) \in L^2(\E;\VV^*)$.  
Since $S$ is smooth and bounded away from zero, we get $\dE\psi\in L^2(\E;\VV^{*})$, which implies $\psi\in C^0(\overline{\E};L^2(\R\times\S))$ and validity of 
\begin{align} \label{eq:derivative_S_norm}
    \dE \|S^{1/2}\psi\|^2_{L^2(\R\times\S)} = 2\la \dE(S\psi), \psi\ra - \la (\dE S) \psi, \psi\ra 
\end{align}
for a.e. $\e \in \E$; see \cite[Ch.~7]{Roubicek} for details.
By appropriate testing of \eqref{eq:vp} and \eqref{eq:strong}, one can  further deduce the  validity of the boundary conditions~\eqref{eq:ibc}. 
Using \eqref{eq:vp} and \eqref{eq:stability} with $v=\psi$, we then see that  $\psi(\emin) \in L^2(\R \times \S)$. 
Since $\psi \in \XX$ with $\psi|_{\Gamma_{in}}=0$ for a.e. $\e \in \E$, we also have
\begin{align} \label{eq:positivity_A}
\la |\s \cdot \n| \, \psi, \psi \ra_{\Gamma_{out}} 
&= \la \s \cdot \n \, \psi, \psi \ra_{\Gamma} 
= 2 \la \s \cdot \nabla_\r \psi,\psi\ra,
\end{align}
which implies $\la \s \cdot \nabla_\r \psi,\psi \ra \ge 0$ and $\psi|_{\Gamma_{out}} \in L^2(\Gamma_{out};|\s\cdot n|)$ for a.e. $\e \in \E$. 
%
By combination of the previous identities and using the notation of Section~\ref{sec:inf_sup}, we arrive at the identity
\begin{align}
    \frac{1}{2} \left(\dE \|S^{1/2}\psi\|^2_{L^2(\R\times\S)} + \la (\dE S) \psi, \psi\ra\right) 
    &= 
    \la \A \psi ,\psi\ra + \la \G \psi,\psi\ra + \la \T \psi,\psi\ra.
\end{align}
From our previous considerations, we know that $\la \G\psi,\psi\ra =0$, $\la \T\psi,\psi\ra \geq 0$, and $\la \A \psi,\psi\ra \ge 0$, which immediately implies
$
-\dE \|S^{1/2}\psi\|^2_{L^2(\R\times\S)} \le C_S \|S^{1/2} \psi\|^2_{L^2(\R \times \S)}.
$ 
By Grönwall's inequality, we then get 
$$
\|S^{1/2}(\e)\psi(\e)\|^2_{L^2(\R\times\S)} \le e^{ C_S (\emax-\e)} \|S^{1/2}(\emax) \psi(\emax)\|^2_{L^2(\R \times \S)} = 0
$$
for all $\e \le \emax$. 
Since $S$ was assumed positive, this yields the desired uniqueness result. 
\end{proof}

\section{Discretization}
\label{sec:disc}

The proof of Theorem~\ref{thm:existence} relies on a weak formulation of the problem and a semi-discretization with respect to energy. Together with a Galerkin approximation in the remaining variables, we obtain an implementable numerical method. 
Similar to \cite{Frank2008}, we here consider a $P_N$-FEM approximation which allows us to utilize much of the analysis presented in \cite{ES2012,ES2016}. 
Let us note that a direct application of \emph{local in angle approximations}, like discrete ordinates or discontinuous Galerkin methods \cite{Bedford2023,Kanschat2009}, would lead to non-conforming approximations of the Fokker-Planck operator, which require certain modifications and a quite different kind of analysis equation which would require a quite different kind of analysis. Investigations in this direction are left for future research.
In the following, we briefly introduce the main ingredients and basic notation, then formally state the method to be used for later discussion, and finally present its convergence analysis.

\subsection{Approximation spaces of the \texorpdfstring{$P_N$}{PN}-finite element method}\label{subsect: PN-Spaces}

Let $\T_h$ denote a geometrically conforming shape-regular partition of $\R$ into tetrahedra, i.e., a typical finite element mesh \cite{Ciarlet_1978},
and let $\XX_h^+$ be the corresponding finite element spaces consisting of continuous piecewise linear functions, and $\XX_h^-$ the space of piecewise constant functions on the mesh $\T_h$, respectively.
Further, let $Y_\ell^m$ with $\ell\geq 0$ and $-\ell\leq m\leq \ell$ denote the spherical harmonics, and recall that they form an orthonormal basis of $L^2(\S)$. 
Some further useful properties of these functions are that $Y_\ell^m$ is even if and only if $\ell$ is even, and that $Y_\ell^m$ are the eigenfunctions of the Laplace-Beltrami operator $-\Delta_\s$ with eigenvalue $\ell(\ell+1)$.
The approximation spaces for the $P_N$-finite element method are then simply defined as
\begin{align*}
    \VV_{h,N}^- &= \{ v_h^- = \sum_{\ell=0\atop \ell\text{ odd}}^N \sum_{m=-\ell}^\ell v_{\ell}^m Y_\ell^m:\, v_\ell^m\in\XX_h^-\},\\
    \WW_{h,N}^+ &= \{ v_h^+ = \sum_{\ell=0\atop \ell\text{ even}}^N \sum_{m=-\ell}^\ell v_l^m Y_\ell^m:\, v_\ell^m\in\XX_h^+\}.
\end{align*}
We further set $\UU_{h,N}=\WW_{h,N}^+\oplus \VV_{h,N}^-$, which is the discrete approximation space for the solution.
Let us recall from \cite{ES2012} the compatibility conditions $\A \WW_{h,N}^+\subset \VV_{h,N}^-$, which is satisfied for order $N$ odd. 

\subsection{The \texorpdfstring{$P_N$}{PN}-finite element scheme}

The fully discrete scheme for \eqref{eq:fp}--\eqref{eq:ibc} is obtained by Galerkin approximation of the semi-discrete scheme~\eqref{eq:var_eq} in the approximation spaces stated above.
We thus set $\psi_{h,N}^M=0$, and look for discrete approximation $\psi_{h,N}^m \in \UU_{h,N}$ for $m=M-1,\ldots,0$, such that
\begin{align}\label{eq:def_discrete_solution}
-\la \bdE(S\psi_{h,N})^m,v_{h,N}\ra + a(\psi_{h,N}^m,v_{h,N}) = \la \bar \q^m,v_{h,N}\ra \qquad \forall v_{h,N}\in\UU_{h,N}.
\end{align}
Let us note that, similar to \eqref{eq:var_eq}, the bilinear form $a(\cdot,\cdot)$ implicitly depends on the iteration index $m$.
For the analysis of the discrete problem, we make an additional assumption, which, however, could be removed by the usual arguments for the analysis of non-conforming Galerkin schemes. 

\begin{assumption} \label{ass:2}
$\E_h = \{\emin = \e^0 < \e_1 < \ldots < \e^M=\emax\}$ with $\e^m = \emin + m \tE$. 
Moreover, $\T_h$ is a simplicial mesh of $\R$, and $\VV_{h,N}^-$, $\WW_{h,N}^+$ are defined as above with $N$ odd. 
Finally, the functions $S$, $T$ are smooth in $\e$ and for each $\e\in\E$ the functions $T(\cdot,\e),S(\cdot,\e)$ are piecewise constant on the mesh $\T_h$.
\end{assumption}

As a direct consequence of this assumption, we see that the operator $\C$ defined before Lemma~\ref{lem:var} is piecewise constant in space, which yields validity of the compatibility condition
\begin{align}
\C^{-1} \A \WW_{h,N}^+ \subset \VV_{h,N}^-.
\end{align}
This allows us to transfer the proof of Lemma~\ref{lem:well_posed_euler} almost verbatim to the discrete setting.
\begin{lemma} 
Under Assumption~\ref{ass:1} and \ref{ass:2}, the scheme \eqref{eq:def_discrete_solution} is well-defined. 
\end{lemma}

In the following section, we derive quasi-optimal error estimates for the proposed method.

\subsection{Error analysis}
\label{sec:error}

In order to work with a norm that is independent of the step size $\tE$ let us redefine, in slight abuse of notation, the norm on $\UU$ as follows:
\begin{align}\label{eq:redef_U_norm}
    \| u \|_\UU^2 = \|u\|_{\T}^2 + \|u^+\|_\partial^2 +\|\A u^+\|_{\T^{-1}}^2.
\end{align}
Let us note that this norm is equivalent to the one defined in \eqref{eq:normU}, 
so all auxiliary results of the previous section can be reused. 
Let $a(\cdot,\cdot)$ be the bilinear form introduced in \eqref{eq:def_bifo}. 
For a given $u \in \UU$, we consider an approximation $u_{h,N} \in \UU_{h,N}$ defined via the discrete variational problem
\begin{align}\label{eq:def_ell_proj}
a(u_{h,N}, v_{h,N}) = a (u, v_{h,N}) \qquad \forall v_{h,N}\in \UU_{h,N}.
\end{align}
With the same reasoning as used in the proof of Lemma~\ref{lem:well_posed_euler}, one can show that the bilinear form $a(\cdot,\cdot)$ is bounded and inf-sup stable on $\UU$ and on the discrete space $\UU_{h,N}$. This leads to the following assertions. 
\begin{lemma}[Ritz projection] \label{lem:projection}
Let Assumptions~\ref{ass:1} and \ref{ass:2} hold. Then for any $u \in \UU$, the system \eqref{eq:def_ell_proj} has a unique solution $u_{h,N} \in \UU_{h,N}$. 
The mapping $\Pi_{h,N} : \UU \to \UU_{h,N}$, $u \mapsto u_{h,N}$ is a projection and satisfies 
\begin{align}\label{eq:quasi-best}
\|u - \Pi_{h,N} u\|_{\UU} \leq C \inf_{v_{h,N}\in\UU_{h,N}} \|u-v_{h,N}\|_{\UU},
\end{align}
with a constant $C$ that is independent of the discretization parameters $\tE,h,N$.
\end{lemma}
The proof is rather standard and follows along the lines of a similar result presented in \cite{ES2012}.
\begin{remark}
Let us emphasize that the bilinear form $a(\cdot,\cdot)$ in \eqref{eq:def_bifo}, and consequently also the projection $\Pi_{h,N}$, will depend on the energy point $\e^m$ in general.
We will write $a^m(\cdot,\cdot)$ and $\Pi_{h,N}^m$ or $\Pi_{h,N}(\e)$ below, if this dependence is important.
We also mention that $\|\cdot\|_\UU$ depends on $\e$ via $T(\e)$, and we use the value of $T(\e)$ when evaluating expressions of the form $\|f(\e)\|_{\UU}$.
\end{remark}
We are now in the position to state and prove our second main result. 
\begin{theorem}[Error estimate]
\label{thrm: error}
Let Assumptions~\ref{ass:1} and \ref{ass:2} hold and $\psi$ be a smooth solution of \eqref{eq:fp}--\eqref{eq:ibc}.
Further assume that $G_i,S,T\in C^2(\overline{\E};L^\infty(\R))$ and that $\tE\leq \frac{c_S}{2 (C_S'+1)}$.
Then there holds
\begin{align*}
    \sup_{0\leq m\leq M} \|\psi(\e^m)-\psi^m_{h,N}\|_{L^2(\R\times\S)} \leq C\big(\tE \|\psi\|_{W^{2,\infty}(\E;\UU)} + \sup_{0\leq m\leq M} &\inf_{v_{h,N}\in\UU_{h,N}} \|\psi(\e^m)-v_{h,N}\|_{\UU} 
    \\
    + &\inf_{v_{h,N}\in\UU_{h,N}} \|\dE\psi(\e^m)-v_{h,N}\|_{\UU}\Big)
\end{align*}
with a constant $C>0$ which does not depend on the discretization parameters $h,N,\tE$.
\end{theorem}
\begin{proof}
The error analysis is based on more or less standard arguments, see e.g. \cite{Thomee}, but for completeness, we present the most important technical details in the following.

\smallskip 
\noindent
\textit{Step 1: Error splitting.} 
Using the abbreviation  $(\Pi_{h,N}\psi)^m =\Pi_{h,N}^m \psi(\e^m)$, we can split the error as 
\begin{align*}
    \psi(\e^m)-\psi^m_{h,N} = [\psi(\e^m)-\Pi_{h,N}^m \psi(\e^m)] + [ (\Pi_{h,N}\psi)^m-\psi^m_{h,N}]. 
\end{align*}
The projection error $\psi(\e^m)-\Pi_{h,N}^m \psi(\e^m)$ can be estimated immediately using \eqref{eq:quasi-best}.
For the discrete error component $e_{h,N}^m=(\Pi_{h,N}\psi)^m-\psi^m_{h,N}$, we will extend the stability estimates of Lemma~\ref{lem:stability}.

\smallskip 
\noindent
\textit{Step 2: Equation for $e_{h,N}^m$.}
Using \eqref{eq:def_ell_proj} and \eqref{eq:def_discrete_solution}, we can see that
\begin{align}
    -\la \bdE(Se_{h,N})^m,v_{h,N}\ra + a^m(e_{h,N}^m,v_{h,N}) = \la \dE(S\psi)^m- \bdE((S\Pi_{h,N}\psi)^m),v_{h,N}\ra.
\end{align}
Using the discrete product rule
$\bdE (S\Pi_{h,N}\psi)^m = (\bdE S^m) (\Pi\psi)^{m+1}+ S^m \bdE((\Pi_{h,N}\psi)^m)$,
we can write
\begin{align}
\dE(S\psi)^m&- \bdE(S\Pi_{h,N}\psi)^m 
= (S'-\bdE S)^m (\Pi_{h,N}\psi)^m 
    + (S' \, (\psi-\Pi_{h,N}\psi))^m \label{eq:rhs_err_1}\\
    &+ (\bdE S^m) \, ( (\Pi_{h,N}\psi)^m - (\Pi_{h,N}\psi)^{m+1}) \label{eq:rhs_err_2}\\
    &+ S^m \Big( \big((\dE\psi)^m-(\dE\Pi_{h,N}\psi)^m\big)+ \big((\dE\Pi_{h,N}\psi)^m- \bdE(\Pi_{h,N}\psi)^m \big)\Big) \label{eq:rhs_err_3},
\end{align}
where $(\dE\psi)^m$ and $(\dE\Pi_{h,N}\psi)^m$ denote the evaluation of the corresponding terms in $\e=\e^m$. The terms on the right-hand side of \eqref{eq:rhs_err_1} can be further estimated by
\begin{align}
    \|(S'-\bdE S)^m) \Pi_{h,N}^m \psi(\e^m)\|_{L^2(\R\times\S)}
    &\leq C \tE \|S''\|_{\infty} \|\psi(\e^m)\|_{\UU}, \label{eq:est_rhs_1}\\
    \|(S' (\psi-\Pi_{h,N}\psi))^m\|_{L^2(\R\times\S)}&\leq C \|S'\|_\infty \inf_{v_{h,N}\in\UU_{h,N}} \|\psi(\e^m)-v_{h,N}\|_{\UU}, \label{eq:est_rhs_2}
\end{align}
where we used \eqref{eq:quasi-best} in the last expression. 
For the remaining terms, we need to investigate in more detail the differentiability properties of the mapping $\e\mapsto\Pi_{h,N}(\e) \psi(\e)$, which we do next.

\smallskip 
\noindent
\textit{Step 3: Derivatives of $\Pi_{h,N}(\e)\psi(\e)$.}
By formally differentiating  \eqref{eq:def_ell_proj}, we observe that 
\begin{align}\label{eq:derivative_ell_proj}
    a(\dE \Pi_{h,N} \psi,v_{h,N}) =a(\dE \psi,v_{h,N}) - a'(\Pi_{h,N} \psi -\psi,v_{h,N}),
\end{align}
for all $v_{h,N} \in \UU_{h,N}$,
where the bilinear form $a': \UU\times\UU\to\RR$ is defined by
\begin{align*}
    a'(u,v) = \la T'(\e)\nabla_\s u,\nabla_\s v\ra + \la G'(\e)\cdot \s\times\nabla_\s u,v\ra \qquad \forall u,v\in\UU.
\end{align*}
Here, $G'$ and $T'$ denote the derivatives of $G$ and $T$ with respect to $\e$. 
By Assumption~\ref{ass:1}, the functions $G'$ and $T'$ are bounded.
Therefore, $\dE\Pi_{h,N}(\e) \psi(\e)\in\UU$. 
By rearranging \eqref{eq:derivative_ell_proj} and using \eqref{eq:quasi-best}, we further see
\begin{align}\label{eq:bound_derivative_ell_proj}
    \| \dE\Pi_{h,N}(\e) \psi(\e) - \dE \psi(\e)\|_{\UU}
    \leq C  \Big(\inf_{v_{h,N}\in\UU_{h,N}} \|\psi(\e)-v_{h,N}\|_{\UU}
     + \inf_{v_{h,N}\in\UU_{h,N}} \|\dE \psi(\e)-v_{h,N}\|_{\UU} \Big),
\end{align}
which we can use to estimate the first term in \eqref{eq:rhs_err_3}.
By differentiating the expression \eqref{eq:derivative_ell_proj} another time with respect to $\e$, we similarly obtain that
\begin{align}\label{eq:sec_derivative_ell_proj}
    a(\dE^2 \Pi_{h,N} \psi,v_{h,N}) =a(\dE^2 \psi,v_{h,N})+2 a'(\dE\psi - \dE \Pi_{h,N} \psi,v_{h,N})+a''(\psi-\Pi_{h,N}\psi,v_{h,N}),
\end{align}
for all $v_{h,N}\in\UU_{h,N}$,
where $a''$ is defined similarly to $a'$, but replacing $T'$ and $G'$ by $T''$ and $G''$, respectively.
From \eqref{eq:sec_derivative_ell_proj} and \eqref{eq:bound_derivative_ell_proj} we then deduce that
\begin{align}
     \|\Pi_{h,N}\psi\|_{W^{2,\infty}(\E;\UU)} \leq C \|\psi\|_{W^{2,\infty}(\E;\UU)}. \label{eq:bound_second_derivative_ell_proj}
\end{align}

\noindent
\textit{Step 4: Putting it all together.} Estimate \eqref{eq:bound_second_derivative_ell_proj} implies that
\begin{align}\label{eq:est_rhs_3}
    \|(\Pi_{h,N}\psi)^{m+1}- (\Pi{_h,N}\psi)^m\|_{\UU} + \| (\dE \Pi_{h,N}\psi)^m- \bdE(\Pi_{h,N}\psi)^m\|_{\UU} \leq \tE C \|\psi\|_{W^{2,\infty}(\E;\UU)},
\end{align}
which we use to estimate the term in \eqref{eq:rhs_err_2} and the second term in \eqref{eq:rhs_err_3}.
By combination of the previous estimates \eqref{eq:est_rhs_1}, \eqref{eq:est_rhs_2}, \eqref{eq:bound_derivative_ell_proj}, and \eqref{eq:est_rhs_3}, we can then bound
\begin{align*}
    \|\dE(S\psi)^m- \bdE(S\Pi_{h,N}\psi)^m\|_{L^2(\R\times\S)}\leq C \Big(\tE\|\psi\|_{W^{2,\infty}(\E;\UU)} +  \inf_{v_h\in\UU_{h,N}} \|\psi(\e^m)-v_h\|_{\UU}  \Big).
\end{align*}
In combination with \eqref{eq:energy} for $\bar \q^m=(\dE(S\psi))^m- \bdE((S\Pi_h\psi)^m)$, we thus obtain
\begin{align}
    \sup_{m}\|(\Pi_{h,N}\psi)^m-\psi^m_{h,N}\|_{L^2(\R\times\S)} \leq C \Big(\tE\|\psi\|_{W^{2,\infty}(\E;\UU)} + \sup_{m} \inf_{v_h\in\UU_{h,N}} \|\psi(\e^m)-v_h\|_{\UU}  \Big).
\end{align}
Together with the previous estimates, this finally proves the bounds of the theorem. 
\end{proof}

\begin{remark}
    The constants in Lemma~\ref{lem:projection} and Theorem~\ref{thrm: error} depend on the bounds for the coefficient functions $S$, $G$ and $T$ and their derivatives. This dependence could be worked out explicitly by careful inspection of all steps in the previous proofs, but it does not provide too much additional insight. Instead of uniform time steps $\tE$, one could also use adaptive time steps in the implementation, which could be included in the analysis with minor modifications; compare to Remark~\ref{rem:local_de}. 
\end{remark}

\section{Numerical results}
\label{sec:num}

In the following numerical tests, we first validate the convergence estimates of Theorem~\ref{thrm: error}, and then illustrate the effect of the Lorentz force on the particle distributions in a typical setting of relevance in applications.
For both test problems, we assume that the particle distribution $\psi$ is homogeneous in the third space direction, which is a common setting in many transport benchmark problems~\cite{Brunner2005}. This facilitates the implementation and allows to consider a spatially two-dimensional domain $\R \subset \RR^2$, while $\S=\{\s \in \RR^3 : |s| = 1\}$ and $\E=(\emin,\emax)$ are defined as before. 
We note that the domain $\R$ used below is merely convex, not $\rho$-convex, but that traces are well-defined for the considered class of functions.
Note that the resulting solutions still have physical meaning in three dimensional space. 
All results of the previous section thus translate verbatim to this setting.

\medskip 
\noindent 
\textbf{Remarks on the numerical realization.}
The implementation of the $P_N$-finite element method for our quasi-two-dimensional model problems can be done as discussed in \cite{ES2012,ES2016}. 
Similar to the integrals involving $\s \cdot \nabla_r$, the additional terms representing  $\G \cdot \s \times \nabla_s \psi$ and $T \Delta_s \psi$ lead to sparse matrices in block-tensor format. 
Every step $m$ of method \eqref{eq:def_discrete_solution} then amounts to the solution of a large relatively sparse linear system. In our numerical tests, we solved this system using a direct sparse solver.

\subsection{Validation of error estimates}
We consider the spatial domain $\R = (0,1)^2$ and the energy interval $\E = (1,2)$. 
The model parameters are defined as $S(x,y,\e) = (1+x^2+y^2)\e^3$, $T(x,y,\e) = (1+x^2+y^2)\e^2$, and $G = (0,0,1)$. The source term $\q$ is chosen such that the solution of \eqref{eq:fp}--\eqref{eq:ibc} is given by
\begin{align}
    \label{eq: sol_manufactured}
    \psi(r,s,\e) = \chi(x,y)f(\e)\sum_{\ell=0}^\infty\sum_{m=-l}^\ell c_\ell^mY_\ell^m(s),
\end{align}
with $r=(x,y)$, $\chi(x,y) = \sin(\pi x)\sin(\pi y)$, $f(\e) = 1-e^{\e-2}$, and $c_\ell^m = \frac{2^{-\ell}}{(2\ell+1)(1+\sqrt{\ell(\ell+1)})}$.
The spherical harmonics $Y_\ell^m$ are normalized such that $\| Y_\ell^m\|_{L^2(\S)}=1$. 
Let us note that the function $\psi$ also satisfies the homogeneous boundary conditions \eqref{eq:ibc} used in our analysis. 

\medskip 
\noindent 
\textbf{Approximation error.}
Before presenting our numerical tests, let us briefly investigate the best approximation error arising in Theorem \ref{thrm: error}. For this purpose we first define the truncated series 
\begin{align}\label{eq:truncated}
    \psi_N(r,s,\e) = \chi(x,y)f(\e)\sum_{\ell=0}^N\sum_{m=-\ell}^\ell c_\ell^m Y_\ell^m(s).
\end{align}
Recalling the definition of the $\UU$-norm in \eqref{eq:redef_U_norm} then allows to estimate the truncation error by
\begin{align} \label{eq:trucn}
    \|\psi(\e^m)-\psi_N(\e^m)\|_\UU &\leq  C 2^{-N}
\end{align}
with a constant $C>0$ that is independent of the discretization parameters. 
We further denote by $\psi_{h,N}(\e) = \psi_{h,N}^+(\e) + \psi_{h,N}^-(\e)$ the discrete approximation for $\psi_N$ in $\UU_{h,N}$ defined by by piecewise linear interpolation resp. piecewise constant projection of $\phi_N^\pm$ with respect to the spatial coordinate. 
By basic error estimates for these finite-element projections, we obtain
\begin{align} \label{eq:proj}
    \|\psi_N(\e^m)-\psi_{h,N}(\e^m)\|_\UU \leq Ch
\end{align}
with a constant $C$ that is again independent of the discretization parameters.
By combination of these estimates, the triangle inequality, and the regularity of $f(\e)$, the estimate of Theorem~\ref{thrm: error} yields
\begin{align}\label{eq:error_bound_manufactured}
   e_{h,N,\tE} := \sup_{0\leq m\leq M}\|\psi(\e^m)-\psi_{h,N}^m\|_{L^2(\R\times\S)}\leq C(\tE+h+2^{-N})
\end{align}
for some constant $C$ that is independent of the discretization parameters $\tE,h$ and $N$. We thus expect first order convergence in $\tE$ and $h$, and exponential convergence with respect to $N$. These rates are optimal in view of the approximation properties of the $P_N$-finite element space and the energy-differencing scheme. 

\medskip 
\noindent 
\textbf{Numerical results.}
In view of the estimate \eqref{eq:error_bound_manufactured}, it makes sense to choose $h$ proportional to $\tE$ and $N$ proportional to $|\log_2(\tE)|$. 
For our numerical tests, we choose $\tE=1/2^j$ with $j=4,\ldots, 8$ 
and set $h=\tE/2$ and $N=2|\log_2(\tE)|-7=1,3,5,7,9$ accordingly.
%
In Table \ref{tab:errors}, we list the errors 
$$e_{\tE} := e_{h(\tE),N(\tE),\tE}$$ 
obtained in our simulations together with the estimated orders of convergence $eoc = \log_2(e_\tE/e_{2\tE})$.
\begin{table} [ht!]
\begin{center}
\caption{Errors $e_{\tE}$ for different values of $\tE$, with estimated order of convergence (eoc).}
    \begin{tabular}{c||c|c|c|c|c}
    $1/\tE$ & $16$ & $32$ & $64$ & $128$ & $256$ \\
    \hline
    $e_{\tE}$ & $0.1411$ & $0.0744$ & $0.0384$ & $0.0195$ & $0.0098$ \\
    \hline
    $eoc$ & $$-\!-\!-$$ & $0.92$ & $0.95$ & $0.98$ & $0.99$\\
    \end{tabular}
    \label{tab:errors}
\end{center}
\end{table}
From our convergence analysis above and the balanced choice of the discretization parameters, we can expect that $e_{\tE} = O(\tE)$, which is in perfect agreement with the actual results obtained in our computations.

\subsection{Effect of the magnetic field}
Our second test problem is motivated by applications in magnetic resonance imaging guided radiotherapy~\cite{Hensel2006,St-Aubin_2015}. 
The region under consideration is irradiated by a primary photon beam that interacts with the tissue and produces secondary electrons with a distribution peaked in the beam direction. These charged particles move through the tissue; they undergo inelastic scattering and absorption, resulting in the deposition of radiation dose, which is the quantity of interest. 
In the presence of a magnetic field, the electrons further experience a Lorentz force resulting in a displacement of the absorbed radiation dose. 

\medskip 
\noindent 
\textbf{Physical background.}
The setup is inspired by \cite{Hensel2006}. We consider a cube of size $L=\qty{30}{\centi\meter}$ consisting of water. The domain is irradiated by an incident beam of primary particles with an energy of about $\qty{10}-\qty{30}{\mega\electronvolt}$. Through inelastic scattering with the background medium, a distribution of secondary electrons with density $q(\r,\s,\e)$ is generated, which has a peak in the direction of propagation of the primary beam. The Fokker-Planck equation~\eqref{eq:fp} describes the steady state distribution $\psi(\r,\s,\e)$ of secondary electrons after propagation, scattering, and absorption in the medium. 
The coefficients $S=S_M$ and $T=T_M+T_\text{Mott}$ 
denote, respectively, the Møller scattering stopping power and Laplace-Beltrami coefficient, and the Mott scattering Laplace-Beltrami coefficient, and they vary strongly as a function of the kinetic energy $\e$; see~\cite[equations (B.2), (B.4) and (B.6)]{Hensel2006} for their expression.
In the presence of a constant magnetic field $B$ of about $\qty{1}{\tesla}$ pointing in the $z$-direction, the electrons experience a Lorentz force, which leads to a coefficient $G=(0,0,G_z)$ as in
\cite{St-Aubin_2015}.
Electrons moving in the $x$-direction will thus also be displaced in the $y$-direction. 
The quantity of interest is the radiation dose, i.e., the amount of energy per volume, deposited within the domain, which is given by $D(\r) =\frac{4\pi T_I}{\delta}\int_0^\infty S_M(r,\e)\Psi(r,\e)\d \e$. 
Here $T_I$ is the irradiation time, $\delta$ the tissue density, and  $\Psi(r,\e) = \frac{1}{4\pi} \int_\S\psi_e(r,\s,\e)\d s$ the angular average of the electron density.
\begin{figure}[ht]
\centering
\includegraphics[width=0.32\textwidth]{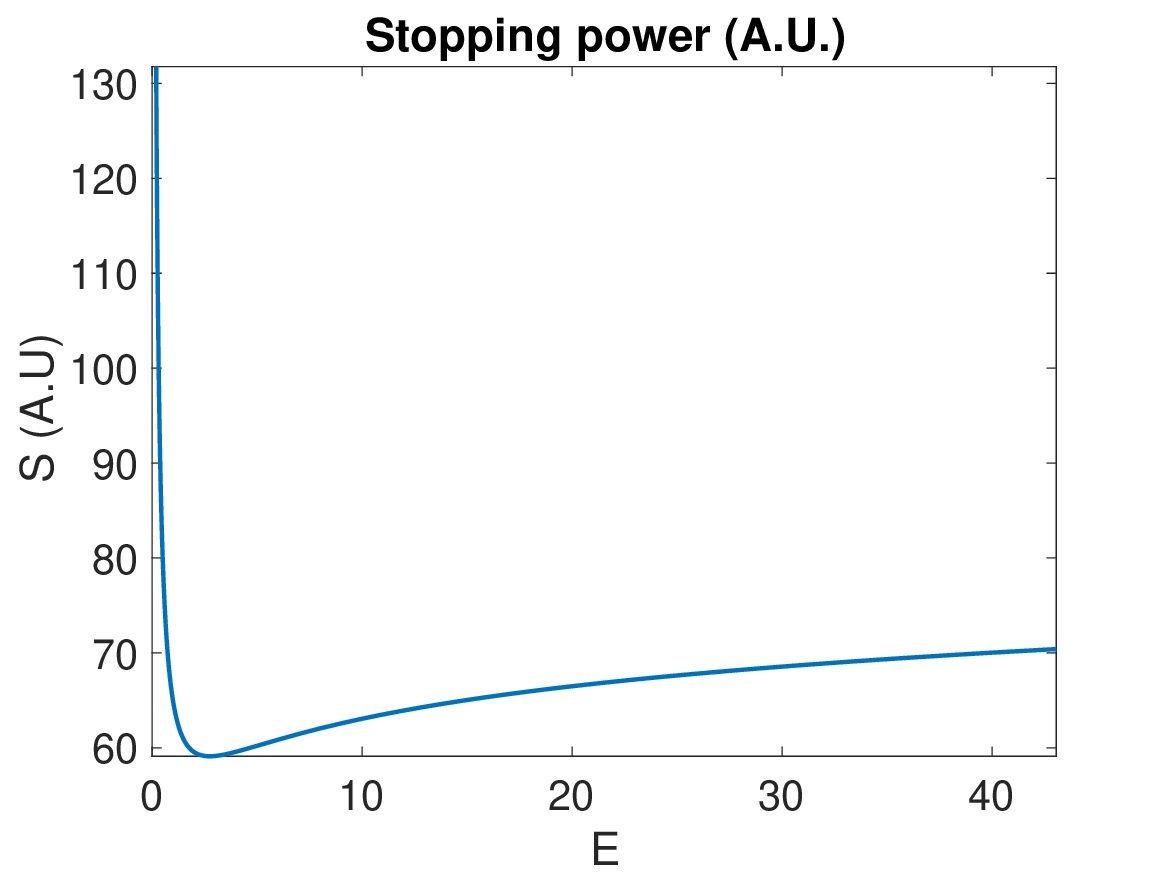}
\hfill
\includegraphics[width=0.32\textwidth]{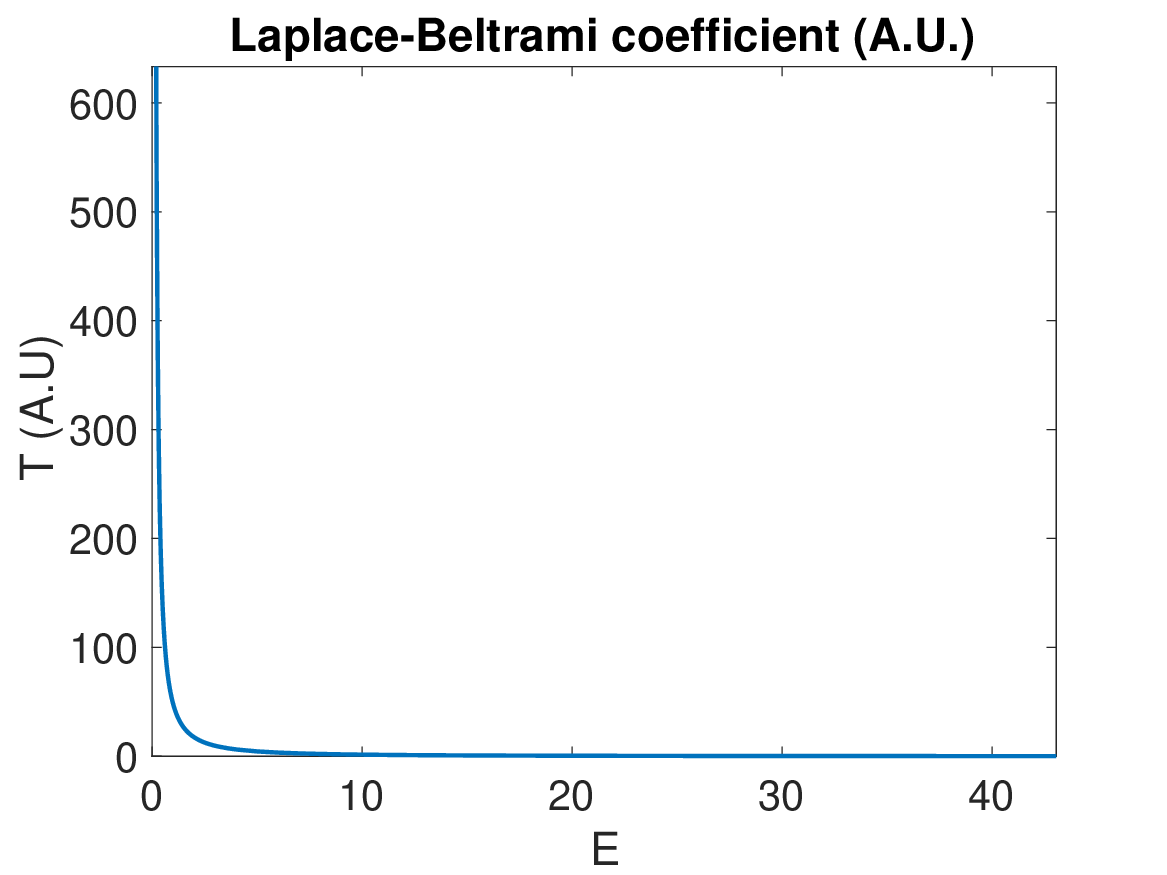}
\hfill
\includegraphics[width=0.32\textwidth]{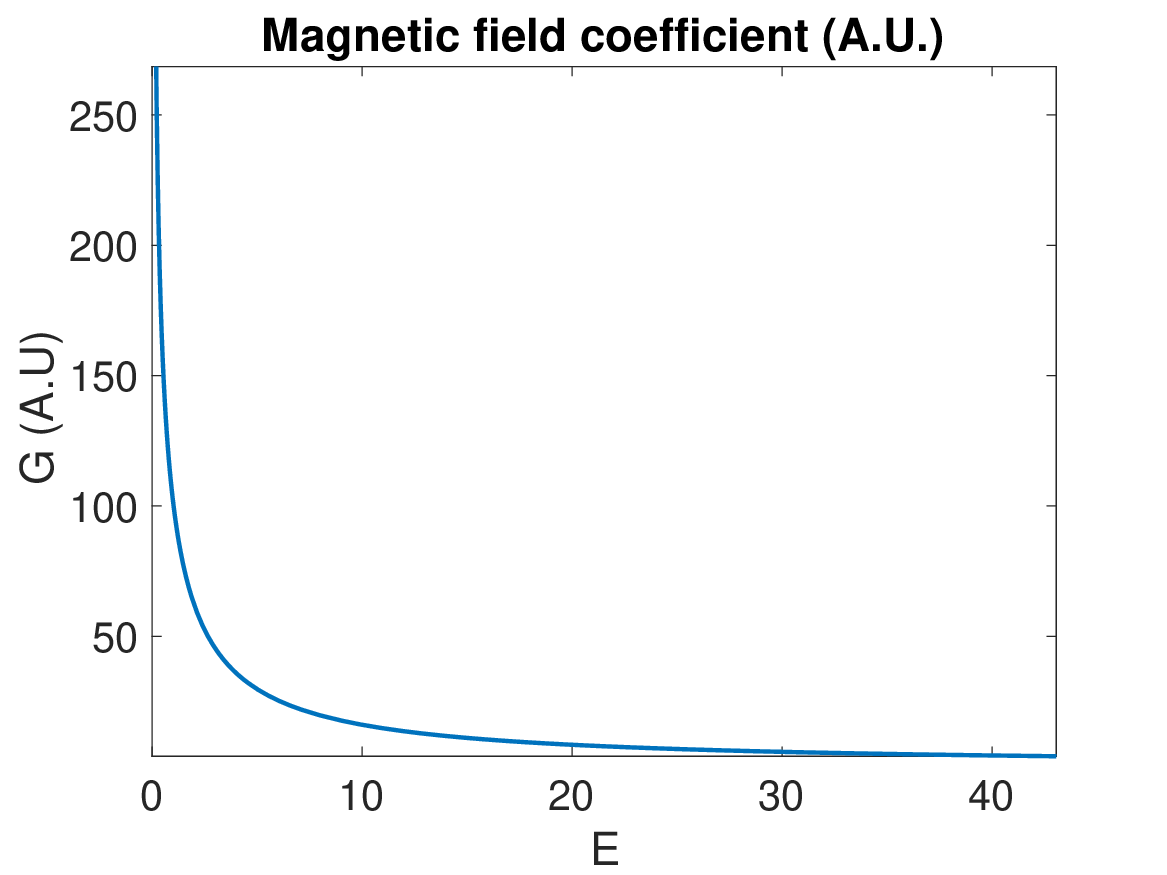}
\caption{Energy dependence of the coefficients $S$, $T$ and $G$ after rescaling in arbitrary units (A.U.).}
\label{fig: Coefficients}
\end{figure}

\noindent 
\textbf{Mathematical model problem.}
After non-dimensionalization and assuming homogeneity of all quantities in the third space direction, we consider the following model problem in our numerical experiment.
The computational domain is chosen as the unit square $\R=(0,1)^2$ and the range of rescaled energies is defined as $\E=(0.2,44)$.
The source density for secondary electrons is given by 
\begin{align*}
    \q(r,s,\epsilon) = e^{-30|r-r_0|^2} e^{-200|s-s_0|^2}e^{-\frac{1}{2}|\e-\e_0|^2},
\end{align*}
with $r_0=(1/2,1/2)$, $\e_0=30$, and $s_0$ denoting the unit vector in positive $x$-direction.
The remaining model parameters have a strong dependence on energy which is depicted in Figure \ref{fig: Coefficients}.
For the dose calculation, we choose the irradiation time such that $\frac{T_I}{\delta}=1$ in rescaled variables.

\medskip 
\noindent 
\textbf{Numerical results.}
For our simulations, we use a spatial mesh $\T_h$ with $18\,432$ spatial elements, a maximal spherical harmonics degree of $N=7$, and $M=400$ uniform steps for discretizing the energy interval. 
In Figure \ref{fig: Dose} we display the computed dose for simulations with and without magnetic field. 
As expected from the physical context, the presence of a magnetic field in the positive $z$-direction results in a counter-clockwise rotation of the electron trajectories and a corresponding displacement of the dose deposited in the medium. 

\begin{figure}[ht]
\centering
\hfill
\includegraphics[width=0.45\textwidth]{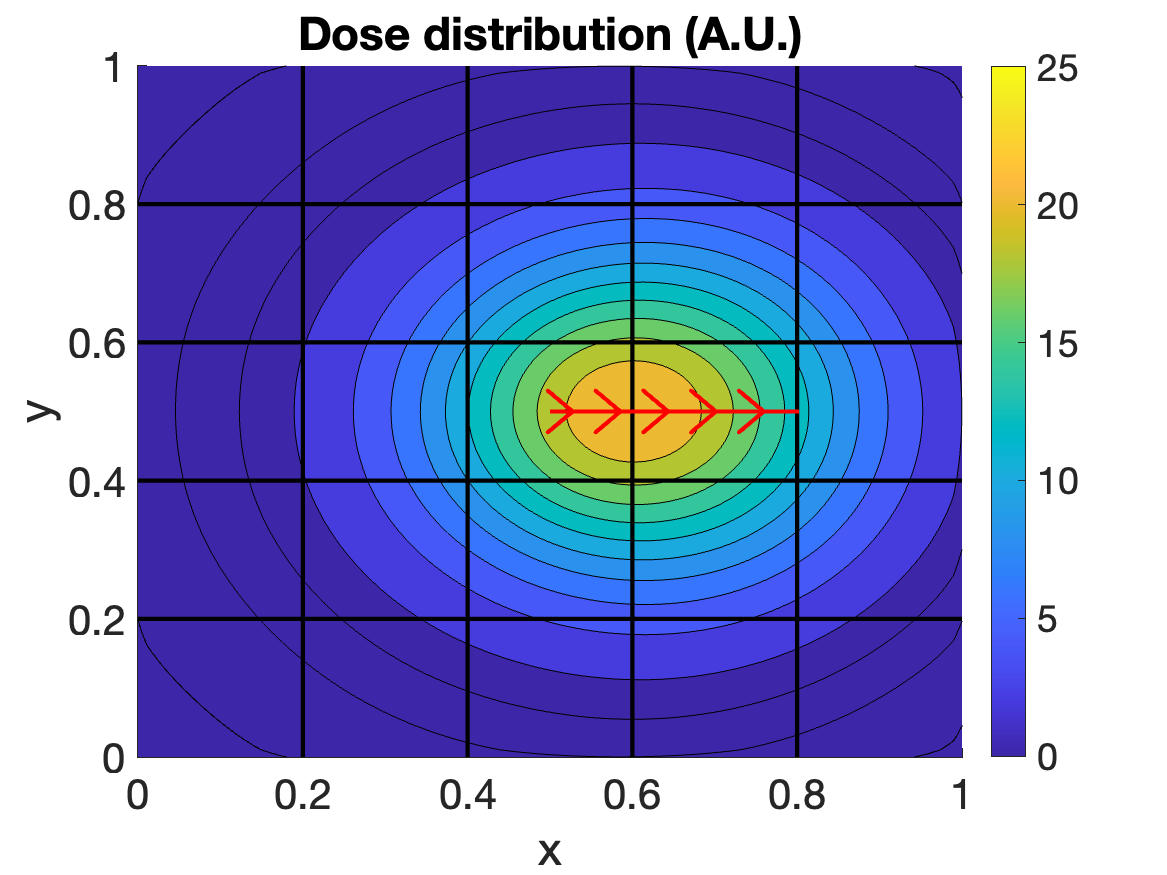}
\hfill
\includegraphics[width=0.45\textwidth]{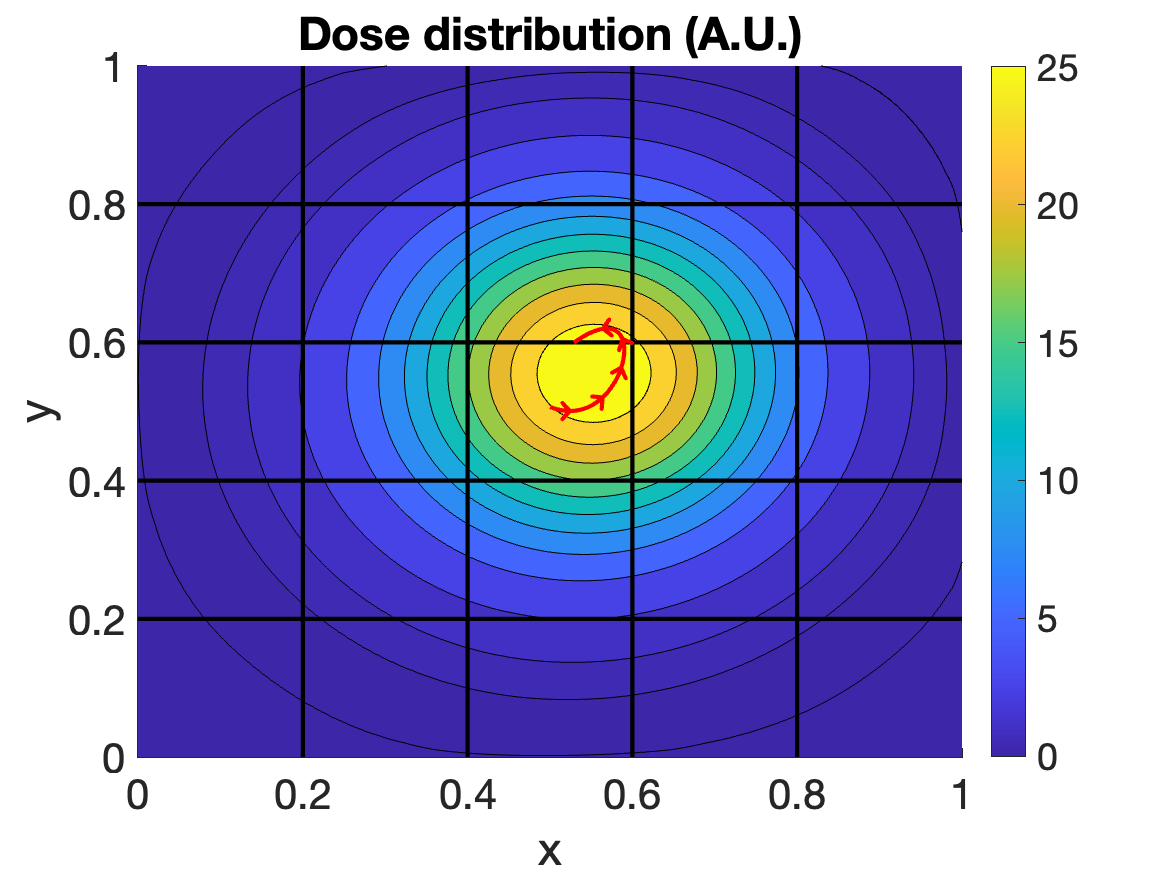}
\hfill
\caption{Dose distribution $D(r)$ (arbitrary units) across the domain without (left) and with (right) magnetic field. The red arrows indicate the trajectory $\e \mapsto {\rm argmax}_{r}\Psi(r,\e)$ of the peak of the dose distribution.}
\label{fig: Dose}
\end{figure}

\noindent 
\textbf{Interpretation in physical terms.}
By rescaling the results to physical quantities, one can see that a magnetic field $B_z=\qty{1}{\tesla}$ results in a higher localization and a shift of the peak of the deposited radiation dose by about $\qty{2.5}{\centi\meter}$, which seems rather significant for the application under consideration.

\begin{funding}
VB and MS acknowledge support by the Dutch Research Council (NWO) via the Mathematics Clusters grant no. 613.009.133.
\url{http://dx.doi.org/10.13039/501100003246}, "Nederlandse Organisatie voor Wetenschappelijk Onderzoek". 
HE was supported by the Austrian Science Fund (FWF) via grant 10.55776/F90.
\end{funding}

\bibliographystyle{plain}
\bibliography{literature}
\end{document}